\documentclass[10pt]{amsart}
\usepackage{amssymb,mathrsfs,graphicx,enumerate}
\usepackage{thmtools}
\usepackage{thm-restate}

\title[ ]{The Contact Geometry of the Spatial Circular Restricted 3-body Problem}
\author{WanKi Cho, Hyojin Jung and GeonWoo Kim}

\newtheorem{theorem}{Theorem}[section]
\newtheorem{lemma}[theorem]{Lemma}
\newtheorem{corollary}[theorem]{Corollary}

\newtheorem{remark}[theorem]{Remark}

\newtheorem{definition}{Definition}[section]

\newcommand{\R}{\mathbb R}

\def\charf {\mbox{{\text 1}\kern-.30em {\text l}}}

\begin{document}

\begin{abstract}
We show that a hypersurface of the regularized, spatial circular restricted three-body problem is of contact type whenever the energy level is below the first critical value (the energy level of the first Lagrange point) or if the energy level is slightly above it. A dynamical consequence is that there is no blue sky catastrophe in this energy range.
\end{abstract}

\maketitle \centerline{\date}


\section{Introduction}
In \cite{AFvKP}, it was proved that the regularized, planar, circular restricted three-body problem is of contact type for energies below and also slightly above the first critical value. Such a result is significant, because it enables the use of holomorphic curve techniques in this classical problem. These techniques have been of great importance in the understanding of the dynamics of Hamiltonian dynamical systems; we mention the work of Floer on the Arnold conjecture, of Hofer and Taubes on the Weinstein conjecture and of Hofer-Wysocki-Zehnder on global surfaces of section. For some of the original works and their improvements using holomorphic curves see \cite{Floer,Hofer:Weinstein,HWZ:GSS,HS:GSS}.
The purpose of this paper is to prove that the spatial, circular restricted three-body problem also has the contact property, hence enabling Floer theoretic techniques to be used in this problem.
As a direct dynamical application, we prove that the spatial, circular restricted three-body problem does not undergo so-called blue sky catastrophes, i.e.,~bifurcations where the period and length of a smooth family of periodic orbits blow up, in the energy range where the contact property holds.

In the restricted three-body problem (RTBP from now on) we consider two massive primaries, which we will refer to as the Earth and Moon, respectively, and a massless satellite that interact with each other via Newtonian gravity.
The Earth and Moon obey Keplerian two-body dynamics, which are well-understood: solutions are described by conic sections. The satellite is influenced by both the Earth and Moon, and this results in more complicated dynamics that are known to be chaotic for almost all mass ratios of the Earth and Moon.
In the circular RTBP, we make the further assumption that the Earth and Moon rotate around each other in circular orbits.
This results in an extra integral, the so-called Jacobi integral, which we will describe now.
Consider the circular RTBP in a frame that is uniformly rotating with constant speed such that the Earth and Moon are fixed on the $x$-axis. Denote the mass of the Earth and Moon by $m_E$ and $m_M$, respectively. Define the mass ratio as $\mu=\frac{m_M}{m_E+m_M}$.
We will write $E=(\mu,0,0)$ and $M=(-1+\mu,0,0)$ for the position of the Earth and Moon in the rotating frame.
The transformation to this rotating frame is achieved by the Hamiltonian vector field coming from the angular momentum function. 
Since the cotangent bundle is trivial in this case, we make identification $T^*(\R^3 \setminus \{ E, M\})= (\R^3 \setminus \{ E, M\}) \times \R^3 $. The position coordinates will be denoted by $q$ and the momentum coordinates in the fiber by $p$.
With this notation in mind, the Hamiltonian that describes the motion of the satellite in $\R^3 \setminus \{ E, M\}$ is given by 
\begin{equation}
\label{abccccc}
\begin{split}
H: (\R^3 \setminus \{ E, M\}) \times \R^3 &\longrightarrow \R,\\
(q,p) &\longmapsto \dfrac{1}{2} \mathopen| p\mathclose|^2 - \dfrac{\mu}{\mathopen| q - M \mathclose|} - \dfrac{1-\mu}{\mathopen|q - E\mathclose|} + p_1 q_2 - p_2q_1.\\
\end{split}
\end{equation}
A derivation of this Hamiltonian can be found in \cite{1} and in \cite{8}.

The energy hypersurface $H^{-1}(c)$ consists of several components depending on the energy level.
Since it will be important to select the correct component, we will introduce some notation. 
Define a projection $\pi$ as
\begin{equation}\label{1bc}
\pi : (\mathbb{R}^3 \setminus \{E, M\}) \times \mathbb{R}^3 \rightarrow \mathbb{R}^3 \setminus \{E,M\}
.
\end{equation}
Level sets of $H$ will be denoted by
\begin{equation} 
\Sigma_c := H^{-1}(c) \qquad \text{for} \quad c \in \mathbb{R}.
\end{equation}
The \textit{Hill's region} is defined as the projection
\begin{equation}
\mathcal{K}_c = \pi(\Sigma_c) \in \mathbb{R}^3 \setminus \{E,M\}.
\end{equation}

For energies $c$ that are lower than the smallest critical value $H(L_1)$ of $H$, both the energy hypersurface $\Sigma_c$ and its projection $\mathcal{K}_c$ consist of three connected components.
Two of the connected components of $\mathcal{K}_c$ are bounded and the other is unbounded. 
The closures of the two bounded regions contain the Moon and Earth, respectively.  We will denote these two components by $\mathcal{K}_c^M$ and $\mathcal{K}_c^E$ such that
\[
M \in \bar{\mathcal{K}}_c^M, \qquad E \in \bar{\mathcal{K}}_c^E
.
\]
Accordingly, we define $\Sigma^M_c$ and $\Sigma^E_c$ as
\[
\Sigma_c^M = \pi^{-1}(\mathcal{K}_c^M) \cap \Sigma_c, \quad \Sigma_c^E = \pi^{-1}(\mathcal{K}_c^E) \cap \Sigma_c
.
\]
For energy $c$ slightly above the first critical value, but below the second smallest critical value, the projection $\mathcal{K}_c$ consists of two connected components. 
Since one of these components is bounded and its closure
contains the Earth and Moon, we refer to this component as  $\mathcal{K}_c^{M,E}$ and define $\Sigma^{M,E}_c$ as
$$
\Sigma_c^{M,E} = \pi^{-1}(\mathcal{K}_c^{M,E}) \cap \Sigma_c
.
$$
Note that $\Sigma_c^{M}$, $\Sigma_c^{E}$ and $\Sigma_c^{M,E}$ are non-compact.
This non-compactness comes from collisions of the satellite with the Moon or Earth. 
These collisions occur whenever $q\to E$ or $q\to M$, and correspond to singular points of the Hamiltonian. To remove these singularities, we will use the regularization introduced by Moser, \cite{8}. We will provide more details when we need them in Section~\ref{sec:moser_reg}.
The upshot is that we can define regularized hypersurfaces, which we denote by $\overline{\Sigma}^M_c$, $\overline{\Sigma}^E_c$ and $\overline{\Sigma}^{M,E}_c$.
In the first two cases, so if $c$ is smaller than $H(L_1)$, the hypersurfaces $\overline{\Sigma}^M_c$ and $\overline{\Sigma}^E_c$, are both diffeomorphic to $ST^*S^3 \cong S^2 \times S^3$.
For $c$ slightly above the first critical value, the regularized component $\overline{\Sigma}^{M,E}_c$ is diffeomorphic to $ST^*S^3 \# ST^*S^3$.

We now state our result.
\begin{restatable}{theorem}{firstthm}
\label{thm:firstthm}
Fix a mass ratio $\mu \in (0,1)$.
Then for $c<H(L_1)$, both regularized energy hypersurfaces $\overline{\Sigma}^E_c$ and $\overline{\Sigma}^M_c$ are of contact type.
Furthermore there exists $\epsilon>0$ such that the regularized energy hypersurface $\overline{\Sigma}^{E,M}_c$ is also of contact type for $H(L_1)<c<H(L_1)+\epsilon$.
\end{restatable}
By abuse of inner product symbol, we use the following notation, $p\cdot\partial_p := \Sigma_i p_i \dfrac{\partial}{\partial p_i}$.
We point out that the Liouville vector field $p\cdot\partial_p$ is not transverse to these level sets due to the magnetic term. Instead, the Liouville vector field that we use to prove this theorem, is inspired by the Moser regularization.
The paper consists of three parts: 
\begin{enumerate}
\item transversality of the Liouville vector field
 for the unregularized problem away from the critical points. This is done in Section~\ref{sec:transversality}.
\item construction of a Liouville vector field near the critical points and gluing this vector field to the Liouville vector field from Section~\ref{sec:transversality}. This is done in Section~\ref{sec:connected_sum}.
\item transversality of the Liouville vector field on the regularization. This is done in Section~\ref{sec:moser_reg}.
\end{enumerate}

Our result has a dynamical corollary. To state this corollary, recall that a blue sky catastrophe is a bifurcation in the sense of \cite{3} where both the period and length blow up.
\begin{corollary}
For $\mu \in (0,1)$, there is no blue sky catastrophe for the Hamiltonian vector field on $\overline{\Sigma}^E_c$, $\overline{\Sigma}^M_c$ for $c<H(L_1)$ and $\overline{\Sigma}^{E,M}_c$ for $H(L_1)<c<H(L_1)+\epsilon$.
\end{corollary}
We will prove this Corollary in Section~\ref{sec:blue_sky}.

\subsection*{Acknowledgements}
During this project, W.C. and G.K. were supported by NRF grant NRF-2016R1C1B2007662. H.J. was supported by the National Research Foundation of Korea(NRF) Grant funded by the Korean Government(MSIT) (No.2017R1A5A1015626)

\section{Construction of restricted spatial three-body problem}
We briefly review the spatial circular RTBP. 
We consider two primaries, which we will call the Earth and Moon, and a massless satellite.
The Earth and Moon will be denoted by $E$ and $M$. We will assume that they have mass $m_E>0$ and $m_M>0$, respectively.
Define the mass ratio as
\[
\mu = \dfrac{m_M}{m_E + m_M} \in (0,1).
\]
The Earth and Moon follow Keplerian dynamics, but we will make the additional assumption that they orbit each other in a circular motion.
After a coordinate transformation and using suitable units, we can arrange that
\[
E(t) = (\mu\cos{t}, \mu\sin{t}, 0), \qquad M(t) = (-(1-\mu)\cos{t}, -(1-\mu)\sin{t}, 0).
\]
The time-dependent Hamiltonian describing the motion of the satellite is the function $H_t : (\mathbb{R}^3 \setminus \{ M(t), E(t) \}) \times \mathbb{R}^3 \rightarrow \mathbb{R}, $ given by
\begin{equation}
H_t(q,p) = \dfrac{1}{2} \mathopen| p\mathclose|^2 - \dfrac{\mu}{\mathopen| q - M(t) \mathclose|} - \dfrac{1-\mu}{\mathopen|q - E(t)\mathclose|}. \nonumber
\end{equation}
We change to a so-called synodical coordinate system, i.e., ~a rotating coordinate system fixing the Earth and Moon to the real axis, so
\[
E = ( \mu, 0, 0), \ \ M = (-(1-\mu), 0, 0).
\]
In this new coordinate system, we obtain an autonomous Hamiltonian $H : (\mathbb{R} ^3 \setminus\{ M, E \}) \times \mathbb{R}^3 \rightarrow \mathbb{R}$ given by
\begin{equation}
H(q,p) = \dfrac{1}{2} \mathopen| p\mathclose|^2 - \dfrac{\mu}{\mathopen| q - M \mathclose|} - \dfrac{1-\mu}{\mathopen|q - E\mathclose|} + p_1 q_2 - p_2q_1. \nonumber
\end{equation}
The term $p_1q_2 -p_2 q_1$, which is minus the angular momentum, is due to the time-dependent coordinate change.
We translate the Moon to the origin and find a new Hamiltonian
\begin{equation}\label{ham}
H(q,p) = \dfrac{1}{2} \mathopen| p\mathclose|^2 - \dfrac{\mu}{\mathopen| q \mathclose|} - \dfrac{1-\mu}{\mathopen|q-e\mathclose|} + p_1 q_2 - p_2(q_1-1+\mu)
\end{equation}
where $e=(1,0,0)$. By completing the squares we get
\begin{equation}\label{divide}
\begin{split}
H(q, p) &= \dfrac{1}{2}((p_1+q_2)^2 + (p_2 - q_1 + 1 -\mu)^2 + p_3^2)\\
 & \ \ \ \ \ \ \ \ \ \ \ \ \ \ \  - \dfrac{\mu}{\mathopen| q \mathclose|} - \dfrac{1-\mu}{\mathopen| q - e \mathclose|} - \dfrac{1}{2}((q_1-1+\mu)^2 + q_2^2).
\end{split}
\end{equation}
This motivates the definition of the effective potential $U$ as
\begin{equation} \label{x}
U (q) = -\dfrac{\mu}{\mathopen| q \mathclose|} - \dfrac{1 - \mu}{\mathopen| q - e \mathclose|} - \dfrac{1}{2}((q_1 - 1 + \mu)^2 + q_2^2).
\end{equation}
The effective potential has five critical points $l^1, l^2, l^3, l^4, l^5$ which are called \textit{Lagrange points}. We order these Lagrange points by their critical values of $U$. 
There are three critical points on the axis between the Earth and Moon. The critical point $l^1$ is located between the Earth and Moon and attains the smallest critical value of $U$.
Using the projection $\pi$ given by \eqref{1bc}, we find a one to one correspondence between critical points of the effective potential and of the Hamiltonian \eqref{divide}.
\begin{equation}
\begin{split}
\pi^{-1} : crit(U) &\rightarrow crit(H) \\
(q_1, q_2, 0) &\mapsto (q_1,q_2,0,-q_2,q_1-1+\mu,0)
\end{split}
\end{equation}
The critical points of the Hamiltonian will be denoted by 
$$L_i := (l^i_{1},l^i_{2},0,-l^i_{2},l^i_{1}-1+\mu,0) \in \mathbb{R}^6\quad \text{for} \ \ i \in \{1,2,3,4,5\}.$$ Denote the values of the Hamiltonian at critical points by 
$$c_i := H(L_i)\quad \text{for }i \in \{1,2,3,4,5\}.$$

In the whole paper we will consider energy surface $H^{-1}(c)$, where $c$ is at most slightly larger than $H(L_1)$. This critical value is at most $-\frac{3}{2}$, so we will assume in particular that $c < -1$.

\section{Below the first critical level in spatial case}
We will use the following form of spherical coordinates in our analysis,
\begin{gather}
q_1 = \rho\cos{\theta}\sin{\varphi} \nonumber \\
q_2 = \rho\sin{\theta}\sin{\varphi} \nonumber \\
q_3 = \rho\cos{\varphi} \nonumber \\
0 \leq \theta < 2\pi, \ 0 \leq \varphi \leq \pi \nonumber.
\end{gather}
Since we will start with energy levels below the first critical value, we can and will assume that the radius $\rho$ is smaller than the distance from the Moon to the first Lagrange point\footnote{The radius of the Moon-component of the Hill's region is bounded by the distance of the Moon to $\ell^1$.}, which is always less than 1.
In these coordinates, the effective potential is
\begin{equation}
U(\rho, \varphi, \theta) = -\dfrac{\mu}{\rho} - \dfrac{1 - \mu}{\sqrt{\rho^2 - 2\rho\cos{\theta}\sin{\varphi}+1}} - \dfrac{1}{2}( (\rho\cos{\theta}\sin{\varphi} - 1 + \mu)^2 +\rho^2\sin^2{\theta}\sin^2{\varphi}).
\end{equation}

For fixed $\rho$, we will find the minimal value of $U$ on this sphere,
which we do by differentiating $U$ with respect to $\theta$.

\begin{lemma}
\label{lemma:minval_pot}
For fixed $\rho$ and $\varphi$, the effective potential $U$ has a minimal value on $\theta = 0$ or $\theta = \pi$.
\end{lemma}

\begin{proof} The first derivative of $U$ with respect to $\theta$ is given by
$$\dfrac{\partial U}{\partial \theta} =(1-\mu)\rho\sin\theta\sin\varphi \left( \dfrac{1}{(\rho^2-2\rho\cos\theta\sin\varphi+1)^{3/2}}-1\right) .$$

The solutions of $\frac{\partial U}{\partial \theta}=0$ satisfy $$\sin\theta = 0,   \ \sin\varphi=0, \text{\ or }\ \frac{1}{(\rho^2-2\rho\cos\theta\sin\varphi+1)^{3/2}}-1=0 \ .$$
Define $A=\rho^2-2\rho\cos\theta\sin\varphi$ and observe that
\begin{equation}
\begin{split}
\dfrac{1}{(\rho^2-2\rho\cos\theta\sin\varphi+1)^{3/2}}-1 &= \dfrac{1}{(A+1)^{3/2}}-1 \\
&= \dfrac{1-(A+1)^3}{(A+1)^{3/2}(1+(A+1)^{3/2})} \\
&= \dfrac{-A^3-3A^2-3A}{(A+1)^{3/2}(1+(A+1)^{3/2})} \\
&= \dfrac{-A(A^2+3A+3)}{(A+1)^{3/2}(1+(A+1)^{3/2})}. \nonumber
\end{split}
\end{equation}
Since we know that $A^2+3A+3 \neq 0$, we see that $A$ must vanish for 
$$ \frac{1}{(\rho^2-2\rho\cos\theta\sin\varphi+1)^{3/2}}-1=0.$$ In other words,
$$\rho(\rho-2\cos\theta\sin\varphi)=0 \Longleftrightarrow  \rho-2\cos\theta\sin\varphi = 0. $$

Therefore the solutions of $\frac{\partial U}{\partial \theta}=0$ are as follows:
\begin{equation}\label{ss}
\rho-2\cos\theta\sin\varphi=0 ,
\end{equation}
\begin{equation}\label{bs}
\theta = 0, \ \pi ,
\end{equation}
\begin{equation}\label{ab}
 \varphi = 0,\ \pi.
\end{equation}

By the way, we will consider the case \eqref{ab} later, which represents the North Pole and South Pole. Note that if  $\varphi \neq 0,\ \pi$, then Equation \eqref{ss} can be rewritten as
\begin{equation}\label{asx}
\cos\theta =\dfrac{\rho}{2\sin\varphi}.
\end{equation}

We split Equation \eqref{asx} into two cases, namely,
\begin{equation}\label{asx1}
 \dfrac{\rho}{2\sin\varphi}=1, \ \ 0<  \dfrac{\rho}{2\sin\varphi}<1.
\end{equation}
In the first case of (\ref{asx1}), $\theta=0$ is a solution of Equation \eqref{asx} which we will consider later. In the second case of (\ref{asx1}), Equation \eqref{asx} has two solutions for $\theta$ which are neither $0$ nor $\pi$. 

We will focus on critical points on subsets with fixed $\varphi$. Each such subset is a circle which is parallel to the plane $\{ \varphi = \pi/2 \}$. 
Now we have to classify critical points on each circle. To do this, it is enough to check the cases $\theta = 0$ and $\theta = \pi$ from \eqref{bs}. Consider the second derivative with respect to $\theta$:
\begin{equation} 
\label{eq:2nd_der}
\begin{split}
\dfrac{\partial^2 U}{\partial \theta^2}
&
=(1-\mu)\rho\sin\varphi \left( \cos\theta\left( \dfrac{1}{(\rho^2-2\rho\cos\theta\sin\varphi+1)^{3/2}}-1\right) \right.
\\
&\phantom{=}
\left.
+\sin\theta \left( -\dfrac{3\rho\sin\theta\sin\varphi}{(\rho^2-2\rho\cos\theta\sin\varphi+1)^{5/2}}\right)\right) 
.
\end{split}
\nonumber
\end{equation}

Note that the second derivative of the effective potential $U$ with respect to $\theta$ is positive at $\theta=\pi$ except $\varphi = 0, \pi$, i.e., $\frac{\partial^2 U}{\partial \theta^2}(\rho,\varphi,\pi) >0$ for all $\rho$ and $\varphi \neq 0, \pi$. 
We compute this as follows:

\begin{equation} \label{normal}
\dfrac{\partial^2 U}{\partial \theta^2}(\rho,\varphi,\pi)=-(1-\mu)\rho\sin\varphi\left( \dfrac{1}{(\rho^2+2\rho\sin\varphi+1)^{3/2}}-1\right) \nonumber.
\end{equation}

Since it is clear that $\rho^2+2\rho\sin\varphi+1>1$, we have $\frac{1}{(\rho^2+2\rho\sin\varphi+1)^{3/2}}-1 <0$. Therefore we get $\frac{\partial^2 U}{\partial \theta^2}(\rho,\varphi,\pi)>0$ for any fixed $\rho$ and $\varphi$. In other words, at $\theta = \pi$ the effective potential $U$ has a local minimum on a circle.\\

We now consider the case $\theta=0$:
\begin{equation} \label{123}
\dfrac{\partial^2 U}{\partial \theta^2}(\rho,\varphi,0)=(1-\mu)\rho\sin\varphi\left( \dfrac{1}{(\rho^2-2\rho\sin\varphi+1)^{3/2}}-1\right).
\end{equation}

Observe that the sign of the second derivative at $\theta = 0$ depends on the sign of $\rho^2-2\rho\sin\varphi$.
First, if $\rho^2 - 2\rho\sin\varphi > 0$, then $\frac{\rho}{2\sin\varphi}>1$ which means there is no real solution of Equation \eqref{asx}. Thus we get $\frac{1}{(\rho^2-2\rho\sin\varphi+1)^{3/2}}-1 < 0$ and the function $U$ has a local minimum at $\theta = 0$ on each circle.
Second, if $\rho^2 - 2\rho\sin\varphi < 0$, then $\frac{\rho}{2\sin\varphi}<1$ which is the second case in \eqref{asx1}. In this case, Equation \eqref{asx} has two additional solutions in $\theta$ which are neither $0$ nor $\pi$. 
Hence there are four critical points and the function $U$ has a local minimum at $\theta=0$ on each circle. Third, if $\rho^2 - 2\rho\sin\varphi = 0$, i.e., $\frac{\rho}{2\sin\varphi}=1$, we know that $\theta =0$ is a solution of Equation \eqref{asx}. 
Note that there are two critical points $\theta=0,\ \pi$ in this case and we have already shown that the function $U$ has a local minimum at $\theta = \pi$ for any $\rho$ and $\varphi$. Since the domain is a circle which is compact without boundary, the function $U$ must have a maximum at $\theta=0$. \hspace*{\fill}

\end{proof}

To summarize, for fixed $\rho$ and $\varphi$, the function $U$ has a minimum at $\theta = 0$ or at $\theta=\pi$. Hence we restrict our domain to the great circle, i.e., $\theta =0$ or $\pi$ and automatically we include the North pole and South pole, which is the case \eqref{ab}. To find a global minimum on the great circle, we differentiate $U$ with respect to $\varphi$.

\begin{theorem}\label{thmU}
For fixed $\rho \in (0,1)$, the effective potential $U$ attains its minimum at $\varphi = \dfrac{\pi}{2},\ \theta = 0$.
\end{theorem}

\begin{proof} 
Because of Lemma~\ref{lemma:minval_pot} we know that for a fixed $\rho$, the effective potential $U$ has a minimum at $\theta = 0$ or $\theta=\pi$. For convenience, we use the standard coordinates and parametrize the great circle as $(\rho\cos \phi,0,\rho\sin\phi)$ where $\phi \in [0,2\pi)$. Now we focus on the restriction of $U$ to the great circle; the restriction of $U$ is given by
$$U|_{(\rho\cos\phi,0,\rho\sin\phi)}= -\dfrac{\mu}{\rho}-\dfrac{1-\mu}{(\rho^2-2\rho\cos\phi+1)^{1/2}}-\dfrac{1}{2}(\rho\cos\phi-1+\mu)^2.$$
The derivative of $U$ with respect to $\phi$ is given by

\begin{equation} \label{2ndfac}
\begin{split}
\dfrac{\partial U}{\partial \phi} \Big|_{(\rho\cos\phi,0,\rho\sin\phi)}&=\dfrac{(1-\mu)\rho\sin\phi}{(\rho^2-2\rho\cos\phi+1)^{3/2}}+(\rho\cos\phi-1+\mu)\rho\sin\phi\\
&=\rho\sin\phi \Bigg(\dfrac{1-\mu}{(\rho^2-2\rho\cos\phi+1)^{3/2}} +\rho\cos\phi-1+\mu\Bigg). 
\end{split}
\end{equation}

It is clear $\phi = 0$ and $\phi = \pi$ are solutions of $\dfrac{\partial U}{\partial \phi} \Big|_{(\rho\cos\phi,0,\rho\sin\phi)} = 0$. To find other critical points, we use the substitution $ t = \cos\phi$ where $0 \leq \phi   \leq \pi$ for the second factor of $\frac{\partial U}{\partial \phi} |_{(\rho\cos\phi,0,\rho\sin\phi)}$. Denote
$$f(t)=\dfrac{1-\mu}{(\rho^2-2\rho t +1)^{3/2}} +\rho t-1+\mu.$$
\\
Observe that
$$f(-1)=\dfrac{1-\mu}{(\rho^2 +2\rho+1)^{3/2}}-\rho -(1-\mu)<0 , $$
$$f(1)=\dfrac{1-\mu}{(\rho^2 -2\rho +1)^{3/2}}+\rho -(1-\mu)>0 ,$$
and $f'(t)$ is computed as follows:
$$f'(t)=\dfrac{3\rho(1-\mu)}{(\rho^2-2\rho t+1)^{5/2}}+\rho >0 .$$
This observation tells us there is exactly one solution of $f(t)$ on the interval $(-1,1)$ and hence the zero of the second factor of Equation (\ref{2ndfac}) is neither $\phi = 0$ nor $\phi = \pi$. Similarly, for $ \pi \leq \phi \leq 2\pi$ we get a zero of the second factor of Equation (\ref{2ndfac}) which is neither $0$ nor $\pi$. Thus $\frac{\partial U}{\partial \phi}$ has four distinct solutions.\\

\indent Next we classify given four critical points. Consider the second derivative of the function $U$ with respect to $\phi$:
\begin{equation}
\begin{split}
\dfrac{\partial^2 U}{\partial \phi ^2} &=\rho\cos\phi \Bigg(\dfrac{1-\mu}{(\rho^2-2\rho\cos\phi+1)^{3/2}} +\rho\cos\phi-1+\mu\Bigg)\\
 &\ \ \ \ \ \ \ \ \ \ \ \ \ \ \ \ \ \ + \rho\sin\phi \  \ \dfrac{\partial}{\partial \phi}\Bigg(\dfrac{1-\mu}{(\rho^2-2\rho\cos\phi+1)^{3/2}} +\rho\cos\phi-1+\mu\Bigg). \nonumber
\end{split}
\end{equation} 
 
The function $U$ has a local minimum on $\phi = 0, \ \pi$ since
$$\dfrac{\partial^2 U}{\partial \phi ^2} \Big|_{\phi=0}=\rho \Bigg(\dfrac{1-\mu}{(\rho^2-2\rho+1)^{3/2}} +\rho-1+\mu\Bigg)>0$$
and
$$\dfrac{\partial^2 U}{\partial \phi ^2} \Big|_{\phi=\pi}=-\rho \Bigg(\dfrac{1-\mu}{(\rho^2+2\rho+1)^{3/2}} -\rho-1+\mu\Bigg)>0.$$
It also has a local maximum on the other critical points because of compactness of $S^1$. Using the Lemma~5.2. of \cite{AFvKP}, the function $U$ attains the global minimum at $\phi = 0$. In spherical coordinates it can be written as
$$\varphi =\dfrac{\pi}{2},\ \theta = 0.$$
\hspace*{\fill}
\end{proof}

\section{Transversality}
\label{sec:transversality}
In this section we will prove a part of the main theorem.
Namely, we will define a Liouville vector field that is transverse to the level set of the Moon component  $\Sigma^M_c$ of the unregularized Jacobi Hamiltonian.
In Section~\ref{sec:moser_reg}, we will complete the proof by showing that this Liouville vector field is also transverse to the regularized set.
\newline

The first step is to compute the derivative of $H$ in \eqref{ham}; the  exterior derivative of $H$ is given by
\begin{equation}\label{dH}
\begin{split}
dH &= p\cdot dp + \dfrac{\mu }{\mathopen| q \mathclose|^3}q\cdot dq + \dfrac{1-\mu}{\mathopen| q-(1,0,0)  \mathclose|^3}(q-(1,0,0) )\cdot dq \\
&\phantom{=}
+ p_1 dq_2 + q_2 dp_1 -p_2dq_1 -(q_1-1+\mu)dp_2.
\end{split}
\end{equation}

If we put the Moon at the origin, the Liouville vector field is given by
\begin{equation} \nonumber
X = q\cdot \dfrac{\partial}{\partial q}.
\end{equation}

Insert the Liouville vector field into (\ref{dH}), and then
\begin{equation}
dH(X) = X(H) = \dfrac{\mu}{\mathopen| q \mathclose|} + (1-\mu)\dfrac{q \cdot (q-(1,0,0))}{\mathopen| q-(1,0,0)\mathclose|^3} + p_1q_2 - p_2q_1.
\end{equation}

\begin{theorem}
\label{part_mainthm}
The Liouville vector field $X$ intersects ${\Sigma^M_c}$ transversely, i.e, $X(H)|_{\Sigma^M_c} $ is positive for $c < H(L_1).$
\end{theorem}

The strategy to prove this theorem is differentiating the effective potential $U$ with respect to $\rho$:
\begin{equation}\label{partialbyrho}
\begin{split}
\dfrac{\partial U}{\partial \rho} &= \dfrac{\mu}{\rho^2} + \dfrac{(1-\mu)(\rho-\cos\theta\sin\varphi)}{(\rho^2-2\rho\cos\theta\sin\varphi+ 1)^{3/2}}\\ 
& \ \ \  - \rho\cos^2\theta\sin^2\varphi + \cos\theta\sin\varphi(1-\mu) - \rho\sin^2\theta\sin^2\varphi.
\end{split}
\end{equation}

Let $d := \mathopen| M - l^1 \mathclose|$, where $M$ is the position of the Moon and $l^1$ is the first Lagrange point and define the ball $B$ as
\begin{equation} \nonumber
B = \{ q \in \mathbb{R}^3 : |q-M| \leq d\}.
\end{equation}
In this points, we need the following lemmas.

\begin{lemma}\label{posU1}
The derivative of U with respect to $\rho$ is positive if  $q \in B \setminus \{M , l^1 \}.$
\end{lemma}

\begin{proof}
One can always find a number $\theta'$ such that
\begin{equation}\label{thetaprime}
\cos\theta' = \cos\theta\sin\varphi,
\end{equation}
and (\ref{partialbyrho}) is reduced to
$$\dfrac{\partial U}{\partial \rho} = \dfrac{\mu}{\rho^2} + \dfrac{(1-\mu)(\rho-\cos\theta')}{(\rho^2-2\rho\cos\theta' + 1)^{3/2}} + \cos\theta'(1-\mu) - \rho\sin^2\varphi.$$
Applying Lemma~5.4 of \cite{AFvKP}, we get that this term is positive. \hspace*{\fill}
\end{proof}

Note that the second derivative of $U$ with respect to $\rho$ is computed by
\begin{equation}\label{partialbyrho2}
\dfrac{\partial^2 U}{\partial \rho^2} = - \dfrac{2\mu}{\rho^3} - \dfrac{(1-\mu)(2\rho^2-4\rho\cos\theta\sin\varphi+3\cos^2\theta\sin^2\varphi-1)}{(\rho^2-2\rho\cos\theta\sin\varphi+1)^{5/2}}-\sin^2\varphi.
\end{equation}
Now we can apply the equation (\ref{thetaprime}) to (\ref{partialbyrho2}) and then the second derivative simplified as follows
\begin{equation}\label{2thetap}
\dfrac{\partial^2 U}{\partial \rho^2} = - \dfrac{2\mu}{\rho^3} - \dfrac{(1-\mu)(2\rho^2-4\rho\cos\theta'+3\cos^2\theta'-1)}{(\rho^2-2\rho\cos\theta'+1)^{5/2}}-\sin^2\varphi.
\end{equation}

\begin{lemma}\label{negU2}
For every $q \in B- \{M\}$, it holds that $\dfrac{\partial^2 U}{\partial \rho^2} \leq -\sin^2\varphi$.
\end{lemma}

\begin{proof}
Define the function $W$ as
\begin{equation}
W(\rho,\theta') := - \dfrac{2\mu}{\rho^3} - \dfrac{(1-\mu)(2\rho^2-4\rho\cos\theta'+3\cos^2\theta'-1)}{(\rho^2-2\rho\cos\theta'+1)^{5/2}}.\nonumber
\end{equation}
By Lemma~5.5 of \cite{AFvKP}, we know the function $W$ is non-positive. Hence we get $\dfrac{\partial^2 U}{\partial \rho^2} \leq -\sin^2\varphi$, using Equation \eqref{2thetap}. \hspace*{\fill}
\end{proof}

Now we can prove Theorem~\ref{part_mainthm} by using Lemma~\ref{posU1} and Lemma~\ref{negU2}.\\
\noindent$\mathit{Proof}$$\mathit{of}$$\mathit{Theorem}$ $~\ref{part_mainthm}$. We use spherical coordinates and so the Liouville vector field turns to $X = \rho\frac{\partial}{\partial \rho}$. We now compute
\begin{equation}
\begin{split}
X(H) &= \rho\dfrac{\partial U}{\partial \rho} + \rho\sin\theta\sin\varphi(p_1 + \rho\sin\theta\sin\varphi) - \rho\cos\theta\sin\varphi(p_2 - \rho\cos\theta\sin\varphi + 1 - \mu)\\
& \geq \rho \dfrac{\partial U}{\partial \rho} - \rho\sin\varphi\sqrt{(p_1+\rho\sin\theta\sin\varphi)^2 + (p_2 - \rho\cos\theta\sin\varphi + 1 - \mu)^2}\\
& = \rho\dfrac{\partial U}{\partial \rho} - \rho\sin\varphi\sqrt{2(H-U)-p_3^2}\\
& \geq \rho\dfrac{\partial U}{\partial \rho} - \rho\sin\varphi\sqrt{2(H-U)} \nonumber,
\end{split}
\end{equation}
where we use the Cauchy-Schwarz inequality for the second step. To show $X(H) >0$, it is enough to check the following
\begin{equation}\label{lastpos}
\rho\left(\dfrac{\partial U}{\partial \rho} - \sin\varphi\sqrt{2(c-U)}\right)\Big|_{\mathcal{K}^M_c} > 0.
\end{equation}
The following inequality is useful for (\ref{lastpos}):
Since we are assuming in this section that $H(L_1) > c$, we have the following with Theorem \ref{thmU}.
\begin{equation}\label{ab2}
U(d, \varphi, \theta) \geq U(d, \dfrac{\pi}{2}, 0)= U(l^1) = H(L_1) > c .
\end{equation}
Note that $U(\rho, \varphi, \theta) \leq c$ for $(\rho, \varphi, \theta) \in {\mathcal{K}^M_c}$. Hence there exists $\tau \in [0, d- \rho)$ such that $U(\rho+\tau, \varphi, \theta) = c$ and we get the following result by Lemma~\ref{posU1} and Lemma~\ref{negU2}.

\begin{equation}
\begin{split}
\left( \dfrac{\partial U(\rho, \varphi, \theta)}{\partial \rho} \right)^2 &= \left( \dfrac{\partial U(\rho + \tau, \varphi, \theta)}{\partial \rho} \right)^2 - \int^\tau_0 \dfrac{d}{dt} \left( \dfrac{\partial U(\rho + t, \varphi, \theta)}{\partial \rho} \right)^2 dt\\
&> -2 \int^\tau_0 \dfrac{\partial U(\rho + t, \varphi, \theta)}{\partial \rho}\dfrac{\partial^2U(\rho+t,\varphi,\theta)}{\partial \rho^2}dt\\
&\geq 2\int^\tau_0 \sin^2\varphi \dfrac{\partial U(\rho+t,\varphi,\theta)}{\partial \rho} dt\\
&= 2\sin^2\varphi(c - U(\rho,\varphi,\theta)) \nonumber.
\end{split}
\end{equation}
This means that Equation~\eqref{lastpos} holds, so we conclude that Theorem~\ref{part_mainthm} holds.

\begin{remark}
\label{rem:extension_transversality}
In the above we have proved transversality of the Liouville vector field for $c < H(L_1)$. 
We now explain how to extend this transversality result on a subset of $H^{-1}(c)$ for any energy level $c$ slightly above $H(L_1)$.

First, take a small ball $B_\delta(L_1)$ around the Lagrange point $L_1$ with some radius $\delta$. If we choose a sufficiently small $\epsilon$, the set $H^{-1}(c) \setminus B_\delta(L_1)$ is still divided into two components by $B_\delta(L_1)$ for $c < H(L_1) + \epsilon$. By Theorem~$\ref{thmU}$, the Hills region of the Moon $\Sigma^M_{H(L_1)}$ has maximal length along the ray $\varphi = \frac{\pi}{2}, \theta = 0$. The maximal distance is smaller than $d := |M - l^1|$, for a sufficiently small $\delta$. That means the Moon component is contained in $B_{d}(M)$. Therefore we can prove $X$ is transverse to the component containing the Moon in $H^{-1}(c) \setminus C_\delta(L_1)$ by the same argument of the proof of Theorem~\ref{part_mainthm}. In the same way, it holds for the component containing the Earth component. This remark will be used in the next section.
\end{remark}

\section{Connected sum}
\label{sec:connected_sum}
So far, we have seen that $H^{-1}(c)$ is of 
contact type whenever $c<H(L_1)$. Now we will show that if $H(L_1)<c<H(L_1)+\epsilon$ for a sufficiently small $\epsilon$, then $H^{-1}(c)$ is also of contact type. This has already been proved for the planar case in section 7 of \cite{AFvKP}. 
For the spatial case, we will apply the same technique and use the same notation as in \cite{AFvKP}. 
We begin by reviewing some of Conley's work from \cite{4}.
We write the Jacobi Hamiltonian as
$$
H(q,p)=\dfrac{1}{2}((p_1+q_2)^2+(p_2-q_1)^2+p_3^2)+V(q),
$$
where $V(q$) is the effective potential given by
$$
V(q)=-\dfrac{1}{2}(q_1^2+q_2^2)-\dfrac{\mu}{|q-M|}-\dfrac{1-\mu}{|q-E|}.
$$
We consider the Taylor expansion of $V(q)$ at a critical point $q^L$,
\begin{equation}\label{159}
V(q)=\tilde{Q}(q-q^L)+R(q-q^L).
\end{equation}
The term $\tilde{Q}$ represents the quadratic part of $V(q)$ and $R$ is the remainder term, i.e., higher order than 2.
This gives us the second order Taylor approximation of $H$ at the first Lagrange point $L_1 :=(q^L,p^L)=(q^L_1,q^L_2,0,-q^L_2,q^L_1,0)$,
$$
H(q,p)=\dfrac{1}{2}((p_1+q_2)^2+(p_2-q_1)^2+p_3^2)+\tilde{Q}(q-q^L)+R(q-q^L).
$$
We combine the quadratic terms of $H$ into the quadratic form $Q$ and find
$$
H(q,p)=Q(q-q^L,p+(q^L_2,-q^L_1,0))+R(q-q^L).
$$
The quadratic form $Q$ is represented by a $6\times6$ matrix, which by abuse of notation we also write as $Q$.
With respect to the ordered basis $(q_1, q_2, p_1, p_2, q_3, p_3)$, this matrix is given by
$$
Q = \dfrac{1}{2} \begin{pmatrix} -2\rho & 0 & 0 & -1 & 0 & 0  \\ 0 & \rho & 1 & 0 & 0 & 0 \\ 0 & 1 & 1 & 0 & 0 & 0 \\ -1 & 0 & 0 & 1 & 0 & 0 \\  0 & 0 & 0 & 0 & \rho & 0 \\  0 & 0 & 0 & 0 & 0 & 1\end{pmatrix},
$$
here $\rho$ stands for the expression 
$$
\rho =  \dfrac{\mu}{|q^L-M|^3}+\dfrac{1-\mu}{|q^L-E|^3}.
$$

We move the first Lagrange point to $l^1 = (0,0,0,0,0,0)$ and we may write $H$ as
$$
H(q,p)=Q(q,p)+R(q).
$$

We will now construct a Liouville vector field near $l^1$.
Define the vector field $Y_{a,b,\gamma}$ as
$$
Y_{a,b,\gamma} = (q_1,q_2,p_1,p_2,q_3,p_3) \begin{pmatrix} a & 0 & 0 & 0 & 0 & 0  \\ 0 & b & 0 & 0 & 0 & 0 \\ 0 & 0 & 1-a & 0 & 0 & 0 \\ 0 & 0 & 0 & 1-b & 0 & 0 \\  0 & 0 & 0 & 0 & \gamma & 0 \\  0 & 0 & 0 & 0 & 0 & 1-\gamma \end{pmatrix} \begin{pmatrix} \partial_{q_1}\\ \partial_{q_2}\\ \partial_{p_1}\\ \partial_{p_2}\\ \partial_{q_3}\\ \partial_{p_3} \end{pmatrix}.
$$
Since the symplectic form is $\omega = dp\wedge dq$, the Lie derivative $\mathcal{L}_Y\omega$ is given by
\begin{equation}
\begin{split}
\mathcal{L}_Y\omega &= d(-aq_1dp_1+(1-a)p_1dq_1-bq_2dp_2+(1-b)p_2dq_2-\gamma q_3dp_3+(1-\gamma)p_3dq_3)\\
&=\omega.
\end{split}
\nonumber
\end{equation}
Therefore the vector field $Y_{a,b,\gamma}$ is Liouville.
Our next task is to show that $Y_{a,b,\gamma}$ is transverse to level sets $H^{-1}(c)$ on small neighborhoods of the critical point. 

\begin{lemma}\label{poslem}
For $a<0$, $b>0$ and $0<\gamma<1$, the vector field $Y_{a,b,\gamma}$ is transverse to $H^{-1}(c)$ in a sufficiently small open ball near $L_1$, where the energy $c$ is contained in $[H(L_1)-\epsilon,H(L_1)+\epsilon] \setminus \{H(L_1)\}$.
\end{lemma}
\begin{proof}
Note that $Y(Q)$ is a quadratic form in $(q,p)$. We also find values $a,b$ and $\gamma$ such that $Y(Q)$ is positive definite.
Split the matrix $Q$ into two block matrices:
\[ 
Q = \dfrac{1}{2} \left( \begin{array}{cccc|cc}
-2\rho & 0 & 0 & -1 &0 &0\\ 
0 & \rho & 1 & 0 & 0 & 0 \\
0 & 1 & 1 & 0 & 0 & 0\\
-1 & 0 & 0 & 1 & 0 & 0\\ \hline
0 & 0 & 0 & 0 & \rho & 0\\
0 & 0 & 0 & 0 & 0 & 1 \end{array} \right)
.
\] 
In this decomposition, the upper block is same as in the planar case and the lower block contributes the term $\gamma\rho q_3^2+(1-\gamma)p_3^2$ to $Y(Q)$. 
Hence we can choose $\gamma$ such that the lower part is positive definite. 
Consequently $Y(Q)$ has positive eigenvalues. On the other hand, the remainder term $Y(R)$ has higher order than $2$. For a sufficiently small neighborhood of the first Lagrange point we can estimate $Y(H)$ by
$$Y(H)=Y(Q+R)=Y(Q)+Y(R) \geq Y(Q)-1/2|Y(Q)|>0.$$ \hspace*{\fill}
\end{proof}

The Liouville form $\alpha_1$ induced by the Liouville vector field $Y$ is given by
$$
\alpha_1=-a(q_1-q^L_1)dp_1-bq_2dp_2-\gamma q_3dp_3+(1-a)p_1dq_1+(1-b)(p_2-q^L_1)dq_2+(1-\gamma)p_3dq_3.
$$
Because of Remark~\ref{rem:extension_transversality}, we already know that for some energy $c \in [H(L_1),H(L_1)+\epsilon]$ the vector field $X=(q-M)\cdot\partial_q$ is Liouville and transverse to the hypersurface $H^{-1}(c)$ outside a small ball $B_{\tilde \delta}(L_1)$. The contact form $\alpha_0$ induced by $X$ is given by
$$
\alpha_0=(M_1-q_1)dp_1-q_2dp_2-q_3dp_3.
$$
The difference of $\alpha_0$ and $\alpha_1$ is written as
\begin{equation}
\begin{split}
\alpha_1-\alpha_0 &=(1-a)((q_1-q^L_1)dp_1+p_1dq_1)+(q^L_1-M_1)dp_1\\
 & \ \ \ \ \ + (1-b)(q_2dp_2+(p_2-q^L_1)dq_2)+(1-\gamma)(q_3dp_3+p_3dq_3).
\end{split}
\nonumber
\end{equation}
Note that this differential form is exact, i.e., $d( \alpha_1 - \alpha_0 ) =d\alpha_1 - d\alpha_0=\omega-\omega=0$, and has a primitive given by
$$G(q,p):=(1-a)(q_1-q^L_1)p_1+(q^L_1-M_1)p_1+(1-b)(p_2-q^L_1)q_2+(1-\gamma)p_3q_3.$$

Now we consider the translation which sends $(q^L_1,q^L_2,0,-q^L_2,q^L_1,0) \mapsto (0,0,0,0,0,0)$ and denote the transformed coordinate by $(q,p)$. Note that $q^L_2 = 0$ since the first Lagrange point is on the axis between the Earth and Moon. We rewrite $G$ as
$$
G=(1-a)q_1p_1+(q^L_1-M_1)p_1+(1-b)p_2q_2+(1-\gamma)p_3q_3.
$$
Denote the two Liouville vector fields by $Z_0$ and $Z_1$ corresponding to $\alpha_0$ and $\alpha_1$, respectively (i.e., $Z_0=X, \  Z_1=Y$). The Hamiltonian vector field $Z_G$ of $G$ is given by
$$
i_{Z_G}\omega=dG
$$
and we have 
$$
Z_1=Z_0+Z_G.
$$
We now construct a Liouville vector field transverse to the hypersurface $H^{-1}(c)$. 
Choose a cut-off function $f$ depending on $q_1+\frac{1}{\rho}p_2$ such that the vector field $Z:=Z_0+Z_{fG}$
\begin{enumerate}
\item equals $Z_0$ for large $q_1$,
\item and equals $Z_1$ for $q_1$ close to $0$.
\end{enumerate}
In other words, $f=1$ near the first Lagrange point and $f=0$ in outside of a small neighborhood of the first Lagrange point. 
We are going to show that the energy hypersurface $H^{-1}(c_1)$ is separated into two connected components by the set $\{q_1+\frac{1}{\rho}p_2=0\}$.

First, we have to check that the singular energy hypersurface $H^{-1}(c_1)$ corresponds to $Q^{-1}(0)$.
If we show that the hyperplane $\{q_1+\frac{1}{\rho}p_2=\delta\}$ intersects the quadric $Q^{-1}(0)$ in 4-sphere when $\delta \neq 0$ and intersects a point when $\delta = 0$, then we can prove that hyperplane $\{ q_1+\frac{1}{\rho}p_2=0 \}$ divides $H^{-1}(c_1)$ into two components.
Rewriting the hyperplane equation by $p_2=\rho(\delta-q_1)$ and putting it into $Q^{-1}(0)$, we get the equation given by
$$\dfrac{1}{2}\Big((p_1+q_2)^2+(p_2-q_1)^2-(2\rho+1)q_1^2+(\rho-1)q_2^2+\rho q_3^2+p_3^2\Big)=0$$
and this equation can be rewritten as
$$(p_1+q_2)^2+(\rho q_1-(\rho+1)\delta)^2+(\rho-1)q_2^2+\rho q_3^2+p_3^2=(2\rho+1)\delta^2.$$
This means that when $\delta=0$, the solution for this equation is just a point. Therefore we have divided the hypersurface into two connected components. We call the component containing the Earth \textit{the Earth component} and the other \textit{the Moon component}.

\indent Now we want to show that the Liouville vector field $Z$ is transverse to level sets $H^{-1}(c)$ where $c \in (H(L_1),H(L_1)+\epsilon)$. Since we have already checked that $Z_0(H)>0$ away from the Lagrange point in Remark~\ref{rem:extension_transversality}, the only remaining thing is interpolating part of the cut-off function $f$. For this, note that $Z(H)$ is given by
\begin{equation}
\begin{split}
Z(H) &= Z_0(H)+Z_{fG}(H)\\
&=dH(Z_0)+\{H,fG\}\\
&= dH(Z_0)-\{fG,H\}\\
&= (1-f)dH(Z_0)+fdH(Z_0)-f\{G,H\}-G\{f,H\}\\
&= (1-f)dH(Z_0)+fdH(Z_0+Z_G)+GdH(Z_f). \nonumber
\end{split}
\end{equation}

It's clear that the first term is non-negative. The second term is also non-negative as $Z_0+Z_G=Z_1$, $dH(Z_1) >0$ by Lemma \ref{poslem}. For the last term, we compute $Z_f$  and $dH (Z_f)$: 
$$
Z_f=f' \cdot \left(\dfrac{\partial}{\partial p_1}-\dfrac{1}{\rho}\dfrac{\partial}{\partial q_2}\right),
$$
$$
dH(Z_f)=f' \cdot \left( 1-\dfrac{1}{\rho}\right)p_1-f' \cdot \dfrac{q_2}{\rho}\left(\dfrac{\mu}{|q-M|^3}+\dfrac{1-\mu}{|q-E|^3}-\rho \right).
$$
Note that the second term of $dH(Z_f)$  vanishes at $q^L$ since $\rho=\frac{\mu}{|q^L-M|^3}+\frac{1-\mu}{|q^L-E|^3}$. In other words, we can estimate $dH(Z_f) \sim f'(1-\frac{1}{\rho})p_1$.

Next we consider the leading order term of $G$, i.e., $(q^L_1-M_1)p_1$, the only 1-degree term of $G$. Observe that the leading order term of $GdH (Z_f)$ is 
$$f'(q^L_1-M_1)\Big(1-\dfrac{1}{\rho}\Big)p^2_1$$
and it is positive in the interpolating region. Notice that the leading order terms in $Z(H)$ are all non-negative and at least one term is positive. Hence, they dominate the higher order term of $Z(H)$. In other words, $Z(H)$ is positive in the interpolating part of the cut-off function $f$. For some more detail, see Section 8.6 of \cite{105}. 

In the next lemma, we extend the above Liouville vector field to the whole hypersurface.

\begin{lemma}\label{compl}
There exists $\epsilon >0$ and a Liouville vector field $\tilde{Z}$ such that $\tilde{Z}$ is transverse to level sets $H^{-1}(E)$ for all $H(L_1)<E <H(L_1)+\epsilon$.
\end{lemma}
\begin{proof}
In the above we have constructed a Liouville vector field $Z$ on the Earth component. 
Since the Liouville vector field is equal to $Z_1$ on a small neighborhood of the first Lagrange point, we can use the same method to the Moon component and get a Liouville vector field $Z'$ on the Moon component.
Because of the symmetry of the cutoff function $f$, these two Liouville vector fields patch together to a Liouville vector field $\tilde{Z}$ on the whole hypersurface. \hspace*{\fill}
\end{proof}

\section{Moser regularization of level sets of $H$}
\label{sec:moser_reg}
In this section we apply Moser's work, \cite{8}, to regularize level sets of the Jacobi Hamiltonian for the spatial RTBP and prove that our Liouville vector field extends to a Liouville vector field that is transverse to this regularized hypersurface. This method has been used before for the planar PTBP in \cite{AFvKP}. See Chapter~4 of \cite{105} for an overview of different regularization schemes.

The computations use stereographic projection and involve both vectors in $\R^3$ and in $\R^4$. In order to streamline some formulas, we will write denote vectors in $\R^3$ by $\vec q=(q_1,q_2,q_3)$ and vectors in $\R^4$ by $\xi=(\xi_0,\xi_1,\xi_2,\xi_3)=(\xi_0;\vec \xi)$. This yields more efficient formulas in this section.
We, however, will make an exception for $E$ and $M$, the positions of Earth and Moon. These are also vectors in $\R^3$, but only serve as parameters and writing $\vec{E}$ will only add clutter.

With this new notation, the Jacobi Hamiltonian is given by
$$
H(\vec p,\vec q)=\dfrac{1}{2}|\vec p|^2-\dfrac{\mu}{|\vec q-M|}-\dfrac{1-\mu}{|\vec q-E|}+p_1q_2-p_2q_1.
$$
As pointed out in the introduction, this Hamiltonian has singularities at $\vec q=M$ and $\vec q=E$, both corresponding to two-body collisions.

\subsection{Regularizing the Hamiltonian} 
We consider the hypersurface $\Sigma^M_c$. To construct the regularization, we will use a new time parametrization and a new Hamiltonian. Namely, we put
$$
s=\displaystyle \int\dfrac{dt}{|\vec q-M|}, \quad K=(H-c)|\vec q-M|,
$$
where the number $c$ is the energy level. In $(\vec q,\vec p)$ coordinates, the function $K$ is given by
$$
K(\vec q,\vec p)=\Bigg( \dfrac{1}{2}|\vec p|^2-\dfrac{\mu}{|\vec q-M|}-\dfrac{1-\mu}{|\vec q-E|}+p_1q_2-p_2q_1-c \Bigg)|\vec q-M|.
$$
We apply the canonical transformation given by $\vec p=-\vec x,\ \vec y=\vec q-M$ to the function $K$ to exchange the role of position and momentum; the transformed Hamiltonian $\tilde{K}$ is given by
\begin{equation}
\begin{split}
\tilde{K}(\vec x,\vec y)&=\dfrac{1}{2}|\vec x|^2|\vec y|-\mu-\dfrac{(1-\mu)|\vec y|}{|\vec y+M-E|}-x_1y_2|y|+x_2(y_1+M_1)|\vec y|-c|\vec y|\\
&=\dfrac{1}{2}(|\vec x|^2+1)|\vec y|-\mu-(1-\mu)\dfrac{|\vec y|}{|\vec y+M-E|}+(x_2y_1-x_1y_2)|\vec y|+x_2 M_1|\vec y|-(c+\dfrac{1}{2})|\vec y|\nonumber.
\end{split}
\end{equation}
We now extend $\tilde K$ to a Hamiltonian on $T^*S^3$ using a symplectic transformation that is induced by the stereographic projection.
This symplectic transformation is given by
\begin{equation}
\label{eq:stereo_symplectic}
\begin{split}
\vec x &=\dfrac{\vec \xi}{1-\xi_0},  \\
\vec y &= \vec \eta (1-\xi_0)+\vec \xi \eta_0,
\end{split}
\end{equation}
where $(\vec x,\vec y)$ represents a point in $T^* \mathbb{R}^3$ and $(\xi,\eta)$ is a point in $T^* S^3\subset T^*\R^4$. 
We have the following formulas for the inverse and for a useful relation,
\begin{equation}
\begin{split}
\xi_0 &= \dfrac{|\vec x|^2-1}{|\vec x|^2+1}, \ \ \ \vec \xi=\dfrac{2\vec x}{|\vec x|^2+1},\\
\eta_0 &= \langle \vec x,\vec y\rangle, \ \ \ \vec \eta=\dfrac{|\vec x|^2+1}{2}\vec y-\langle \vec x,\vec y \rangle \vec x,\\
|\eta|&= \dfrac{(|\vec x|^2+1)|\vec y|}{2} = \dfrac{|\vec y|}{1-\xi_0}.
\end{split}
\end{equation}
By plugging in these formulas into $\tilde K$, we obtain the transformed Hamiltonian $F$ on $T^*S^3$:
$$
F(\xi,\eta)=|\eta|\Bigg(1-\dfrac{(1-\mu)(1-\xi_0)}{|(1-\xi_0)\vec{\eta}+\eta_0\vec{\xi}+M-E|} +(1-\xi_0)(\xi_2\eta_1-\xi_1\eta_2)+\xi_2 M_1-(c+\dfrac{1}{2})(1-\xi_0) \Bigg)-\mu.
$$
For convenience, let us define a function $f(\xi,\eta)$ by
$$f(\xi,\eta)=1-\dfrac{(1-\mu)(1-\xi_0)}{|(1-\xi_0)\vec{\eta}+\eta_0\vec{\xi}+M-E|} +(1-\xi_0)(\xi_2\eta_1-\xi_1\eta_2)+\xi_2 M_1-(c+\dfrac{1}{2})(1-\xi_0)$$
and consider the new Hamiltonian $Q(\xi,\eta)$ given by
$$
Q(\xi,\eta)=\dfrac{1}{2}|\eta|^2 f(\xi,\eta)^2.
$$
Note that the level set  $H^{-1}(c) = K^{-1}(0)$ is compactified to the level set $Q^{-1}(\frac{1}{2}\mu^2)$.
As the Hamiltonian $Q$ is smooth near this level set with the same (up to time parametrization) dynamics, we view $Q^{-1}(\frac{1}{2}\mu^2)$ as the regularized problem.

\subsection{Transversality near the Moon}
The goal of this subsection is to show that the level set $Q^{-1}(\frac{1}{2}\mu^2)$ is of contact type. In other words, we check that the Liouville vector field $X=\eta \cdot \partial_\eta$ is transverse to the level set $Q^{-1}(\frac{1}{2}\mu^2)$ near the Moon, i.e., over points $(\xi,\eta)$ satisfying $|\eta|(1-\xi_0)<\epsilon$. Of course there are two connected components but we only consider the Moon component. For $\mu > 0$, $X(Q)$ is given by
\begin{equation}
\begin{split} \nonumber
X(Q)&=|\eta|^2f(\xi,\eta)^2+|\eta|^2f(\xi,\eta)\eta\cdot \partial_\eta f(\xi,\eta) \\
&=2Q+|\eta|^2f(\xi,\eta)\Bigg(-\eta\cdot\partial_\eta\dfrac{(1-\mu)(1-\xi_0)}{|(1-\xi_0)\vec{\eta}+\eta_0\vec{\xi}+M-E|}+(1-\xi_0)(\xi_2\eta_1-\xi_1\eta_2) \Bigg).
\end{split}
\end{equation}
It is clear that the first term is positive. The strategy of estimating the second term is to obtain a lower bound on $|f(\xi,\eta)|$ and an upper bound on $|\eta|$ from the level set condition $Q^{-1}(\frac{1}{2}\mu^2)$. We will bound $|f(\xi,\eta)|$ by following inequality 
\begin{equation}
\begin{split}
|f(\xi,\eta)| &= \Bigg|1-\dfrac{(1-\mu)(1-\xi_0)}{|(1-\xi_0)\vec{\eta}+\eta_0\vec{\xi}+M-E|} +(1-\xi_0)(\xi_2\eta_1-\xi_1\eta_2)+\xi_2 M_1-(c+\dfrac{1}{2})(1-\xi_0)\Bigg|\\
&\geq \Bigg|1-\dfrac{(1-\mu)(1-\xi_0)}{|(1-\xi_0)\vec{\eta}+\eta_0\vec{\xi}+M-E|}-(c+\dfrac{1}{2})(1-\xi_0)\Bigg|-\Bigg|(1-\xi_0)(\xi_2\eta_1-\xi_1\eta_2)+\xi_2M_1\Bigg|\\
& \geq 1+(1-\xi_0)\Bigg(|c|-\dfrac{1}{2}-\dfrac{1-\mu}{|(1-\xi_0)\vec{\eta}+\eta_0\vec{\xi} + M - E|}\Bigg)-(1-\xi_0)|\eta||\xi|-|\xi||M| ,
\end{split}
\nonumber
\end{equation}
where we have assumed that $c$ is negative.
We are making this assumption, since we are concerned with level sets $H^{-1}(c)$ with $c<H(L_1)+\epsilon<0$.
Because $|M|=1-\mu,|\xi|=1$ and $(1-\xi_0)|\eta|<\epsilon$, we find that
\begin{equation}
\begin{split}
|f(\xi,\eta)| &\geq 1 -\epsilon-(1-\mu) +(1-\xi_0)\Bigg(|c|-\dfrac{1}{2}-\dfrac{1-\mu}{|(1-\xi_0)\vec{\eta}+\eta_0\vec{\xi} + M - E|}\Bigg)\\
&\geq \dfrac{\mu}{2}.
\end{split}
\nonumber
\end{equation}
The last inequality follows from the fact that the first three terms give $\mu-\epsilon$ and the number $(1-\xi_0)$ is always non-negative. 
We can simplify the last term as
\begin{equation}\label{1a}
(1-\xi_0)\Bigg(|c|-\dfrac{1}{2}-\dfrac{1-\mu}{|q-E|}\Bigg).
\end{equation}
By the triangle inequality 
$$
|q-M|+|q-E|\geq |M-E|=1,
$$
we get
$$
|\vec q-E|\geq 1-\epsilon.
$$
Note that for every $\mu \in [0,1]$ we know that $H(L_1) \leq -\frac{3}{2}$. Hence we get a lower bound on the term \eqref{1a} by
\begin{equation} \label{rhs}
(1-\xi_0)\Bigg(|c|-\dfrac{1}{2}-\dfrac{1-\mu}{|\vec q-E|}\Bigg) \geq (1-\xi_0)\Big(\dfrac{3}{2}-\dfrac{1}{2}-\dfrac{1-\mu}{1-\epsilon}\Big).
\end{equation}
The right hand side of \eqref{rhs} approaches $(1-\xi_0)\mu>0$ as $\epsilon$ approaches $0$.
Furthermore, it is also increasing as $\epsilon$ tends to $0$.
Hence it follows that there exists a number $\epsilon'$ such that if $\epsilon < \epsilon'$, then the right hand side of \eqref{rhs} is bigger than $0$. Thus we get
$$
|f(\xi,\eta)| \geq \mu-\epsilon -0 \geq \dfrac{\mu}{2}.$$
In particular, we can now bound the range of $\eta$,
$$
\dfrac{1}{2}\mu^2=Q(\vec q,\vec p)=\dfrac{1}{2}|\eta|^2 |f(\xi,\eta)|^2 \geq \dfrac{1}{2}|\eta|^2 \dfrac{\mu^2}{4},
$$
$$
\text{i.e.,} \ |\eta| \leq 2.
$$
Let us show that $X(Q)$ is positive. Observe that
$$
X(Q) \geq 2Q-|\eta|^2|f(\xi,\eta)| \Bigg| -\eta\cdot\partial_\eta\dfrac{(1-\mu)(1-\xi_0)}{|\vec{\eta}(1-\xi_0)+\vec{\xi}\eta_0+M-E|}+(1-\xi_0)(\xi_2\eta_1-\xi_1\eta_2) \Bigg|.
$$
Note that $|\eta| \leq 2$, $|\eta||f(\xi,\eta)|= \sqrt{2Q} = \mu$. 
By the triangle inequality, we get
$$
X(Q) \geq 2Q-2\mu\Bigg(|\eta| \Bigg|\partial_{\eta} \dfrac{(1-\mu)(1-\xi_0)}{|\vec{\eta}(1-\xi_0)+\vec{\xi}\eta_0+M-E|}\Bigg| +|(1-\xi_0)(\xi_2\eta_1-\xi_1\eta_2)|\Bigg).
$$
The term $\partial_\eta f(\xi,\eta)$ is smooth on the region $|\vec q-M| \leq \epsilon$, since this region is away from the singularity at $E$. Hence we can bound this term by $C$, and we continue to estimate $X(Q)$ as follows
$$
X(Q) \geq \mu^2-2\mu\epsilon(1+(1-\mu)C).
$$
By choosing $\epsilon$ sufficiently small, we see that $X(Q)=\eta\cdot\partial_\eta Q$ is positive.

Finally, the symplectic map~\eqref{eq:stereo_symplectic} composed with the switch map sends the Liouville vector field $\vec q\cdot\partial_{\vec q}$ to the Liouville vector field $\eta \cdot \partial_\eta$. Thus we have obtained a Liouville vector field defined near the whole regularized level set and showed transversality using the computations from the previous sections.
As a result, the regularized hypersurface is of contact type and diffeomorphic to $ST^*S^3$ when the energy is below the first critical value. If the energy is slightly above the first critical value, then the regularized energy hypersurface is also of contact type and it is diffeomorphic to the connected sum $ST^*S^3\# ST^*S^3$.
This concludes the proof of Theorem~\ref{thm:firstthm}.

\section{Blue sky catastrophes}
\label{sec:blue_sky}
In this section we will investigate some dynamical consequences of the contact-type condition; we will see that no blue sky catastrophes can occur in the range of energy-parameters where Theorem~\ref{thm:firstthm} applies.

Let us start by explaining what a blue sky catastrophe is.
Simply put, if we are given a smooth $1$-parameter family of periodic orbits $\{ \gamma_s \}_{s\in[0,s_0)}$, then we say that this family undergoes a blue sky catastrophe at $s=s_0$ if both the period and length go to infinity as $s\to s_0$.

Since we will consider the Reeb vector field, which is non-vanishing, on a compact manifold, we will define a blue sky catastrophe with these additional assumptions as follows.
\begin{definition}
A smooth $1$-parameter family of non-vanishing vector fields $\{ X_{s} \}_{s\in [0,1]}$ on a compact manifold has a blue sky catastrophe at $s=1$ if there is a smooth $1$-parameter family of periodic orbits $\gamma_s$ with $s\in [0,1)$ such that the period of $\gamma_s$ goes to infinity as the parameter $s$ goes to 1.
\end{definition}
The extra assumptions imply that the length also goes to infinity in this case.
A Hamiltonian example of such a period and length blow-up in a $1$-parameter family occurs for instance in the transition from the geodesic flow to the horocycle flow on the unit cotangent bundle of a higher genus surface. See the introduction of \cite{BFvK} for a more detailed discussion of this example.\\

For the following theorem, we need some assumptions and notations. Let $(M,\omega=d\lambda)$ be an exact symplectic manifold and $H$ be a Hamiltonian function in $ C^{\infty} ( M \times [0,1], \mathbb{R})$. Denote $H_r = H( \cdot, r) \in C^{\infty} ( M , \mathbb{R})$. With this notation, we can treat $H$ as a one parameter family of autonomous Hamiltonian functions $H_r$ for $r \in [0,1]$. 

Assume that for every $r \in [0,1]$, $0$ is a regular value of $H_r$  and the level set $H^{-1}_r(0)$ is connected. Moreover, $H^{-1}(0)$ is supposed to be compact such that $\{H^{-1}_r(0)\}_{s\in [0,1]}$ is a smooth one parameter family of closed, connected submanifolds of $M$. We have the Hamiltonian vector field of $H_r$ for $r \in [0,1]$, denoted by $X_{H_r}$, satisfying the condition
$$dH_r = \omega(\cdot, X_{H_r}).$$

\begin{theorem}
\label{thm:no_blue-sky}
Suppose that $(\gamma_r,\tau_r)$ for $r\in[0,1)$ is a smooth family of periodic Reeb orbits $\gamma_r$ with period $\tau_r$.
If there is a one parameter family of contact form $\lambda|_{H_r^{-1}(0)}$ with $r\in[0,1]$,  then there exists $\tau_1 \in (0,\infty)$ such that $\tau_r$ converges to $\tau_1$.
\end{theorem}

We will follow the idea in chapter 7 of \cite{105} to prove this theorem. Before proving Theorem \ref{thm:no_blue-sky}, we want to make a remark: consider the Rabinowitz action functional on the manifold $M$,
$$
\mathcal{A}^H(\gamma,\tau)=\displaystyle \int_{S^1}\gamma^*\lambda-\tau\displaystyle \int H(\gamma(t))dt,
$$
where $\gamma \in C^{\infty}(S^1,M)$ and $\tau \in (0,\infty)$. If $(\gamma,\tau)$ is a critical point of $\mathcal{A}^H$, then 
$$
\partial_t\gamma=\tau X_H(\gamma) \text{\ \ \ and\ \ \ } H(\gamma)=0.
$$
In other words, a critical point $( \gamma, \tau)$ of $
\mathcal{A}^H$ is a periodic orbit of $X_H$ with period $\tau$.

\begin{proof}
Note that the Hamiltonian flow is given by a reparametrization of the Reeb flow of the contact form $\lambda|_{H^{-1}_r(0)}$ and their images are same. Indeed, the Reeb vector field $R_r$ of $\lambda|_{H^{-1}_r(0)}$ is parallel to the Hamiltonian vector field $X_{H_r}|_{H^{-1}_r(0)}$. 
Observe that if two vector fields are parallel each other, then there exist smooth functions $f_r : {H^{-1}_r(0)} \rightarrow \mathbb{R}$ such that 
$$R_r = f_r X_{H_r} \vert_{H^{-1}_r(0)}.$$
By the compactness of $H^{-1}(0)$, there exists $c >0 $ such that
$$\frac{1}{c} \leq \vert f_r(x)\vert \leq c\ \ \ \ \ \text{for} \ \ r  \in [0,1],\ x \in H^{-1}_r(0). $$
Thus we can get a smooth extension of $f_r$ such that
 $\bar{f} : M \times [0,1] \rightarrow \mathbb{R}\text{\textbackslash} \{0\}$ and
$$\bar{f} ( \cdot, r) \vert_{H^{-1}_r(0)} = f_r.$$
By replacing $H$ by $\bar{f} \cdot H$, we may assume that
\begin{equation}\label{R_r}
R_r=X_{H_r}|_{H^{-1}_r(0)}.
\end{equation}
In the same way, the original family of periodic orbits $\gamma_r$ gets reparametrized. However, one can notice that the reparametrization of flow does not affect to convergence of period because of the compactness of $H^{-1}(0)$.\\ 

Consider the family of functionals for $r \in [0,1]$
$$\mathcal{A}^{H_r} : C^\infty(S^1,M)\times (0,\infty) \rightarrow \mathbb{R}.$$
We now compute the action of $\mathcal{A}^{H_r}$ at the critical point $(\gamma, \tau)$ with (\ref{R_r}) as follows
$$
\mathcal{A}^{H_r}(\gamma,\tau) = \displaystyle \int^1_0 \lambda (\tau X_{H_r}(\gamma))dt= \tau\displaystyle \int^1_0 \lambda(R_r)dt = \tau.
$$
Notice that the period $\tau$ of the periodic orbit $\gamma$ is the action value of $\mathcal{A}^{H_r}$. 
Let $(\gamma_r, \tau_r)$ for ${r \in [0,1)}$ be a smooth family of periodic orbits. Differentiating $\tau_r$ with respect to the $r$-parameter, we get 
\begin{equation}\label{partial_r}
\begin{split}
\partial_r \tau_r &= \dfrac{d}{dr}(\mathcal{A}^{H_r}(\gamma_r,\tau_r))\\
&=(\partial_r \mathcal{A}^{H_r})(\gamma_r,\tau_r)\\
&= -\tau_r\displaystyle \int_{S^1}(\partial_r H_r)(\gamma_r)dt.
\end{split}
\end{equation}
Note that we use the fact that $(\gamma_r, \tau_r)$ is a critical point of $\mathcal{A}^{H_r}$ for the second equation. 
By the fact that $H^{-1}(0)$ is compact, there exists $k>0$ such that
\begin{equation}\label{Kappa}
\Big| \partial_r H_r|_{H^{-1}_r(0)} \Big| <k\ \ \ \ \ \ \text{for all}\ \ r \in [0,1].
\end{equation}
Using (\ref{partial_r}) and (\ref{Kappa}) we have the estimate
$$| \partial_r \tau_r |  < k \tau_r.$$
Hence, if $0 \leq r_1 < r_2 <1$, then
$$e^{-k(r_2-r_1)}\tau_{r_1} < \tau_{r_2} < e^{k(r_2-r_1)}\tau_{r_1}.$$
This completes that $\tau_r$ converges, when $r$ goes to $1$.
\hspace*{\fill}
\end{proof}

We are now in position to get the conclusion of this section.

\begin{corollary}
Suppose that $\gamma_s$ is a smooth family of periodic orbits on the regularized energy hypersurface $\overline{\Sigma}^E_c$, $\overline{\Sigma}^M_c$ or $\overline{\Sigma}^{E,M}_c$ of the Hamilton vector field $X_{H_{\mu(s),c(s)}}$ for some smooth family $s\mapsto \mu(s),c(s)$.
Assume that $H_{\mu(s),c(s)}(\gamma_s)\neq H_{\mu(s),c(s)}(L_1(\mu(s) ) )$ and that $H_{\mu(s),c(s)}(\gamma_s) < H_{\mu(s),c(s)}(L_1(\mu(s) ) ) +\epsilon(\mu(s))$, where $\epsilon(\mu(s))$ is the small parameter from Theorem~\ref{thm:firstthm}. Then $\gamma_s$ has no blue sky catastrophe.
\end{corollary}
By Theorem~\ref{thm:firstthm} we know that such the hypersurfaces in such a family are all of contact-type. Hence Theorem~\ref{thm:no_blue-sky} implies that there is no blue sky catastrophe.

\end{document}